\documentclass{amsart}

\usepackage[utf8]{inputenc}
\usepackage[T1]{fontenc}
\usepackage{amsmath}
\usepackage{amssymb}
\usepackage{amsthm}
\usepackage{oldgerm}
\usepackage{multicol}
\usepackage{verbatim}
\usepackage{color}
\usepackage{ulem}
\usepackage{enumitem}
\usepackage{mathrsfs}

\usepackage[pdftex,colorlinks,urlcolor=blue,pdfstartview=FitH,
]{hyperref}

\newcommand{\Rm}{\mathbb{R}}

\newcommand{\RR}{\mathbb{R}}

\newcommand{\mC}{\ensuremath{\mathcal{C}}}

\newcommand{\mF}{\ensuremath{\mathcal{F}}}
\newcommand{\mP}{\ensuremath{\mathcal{P}}}

\newcommand{\Zm}{\ensuremath{\mathbb{Z}}}

\newcommand{\mD}{\ensuremath{\mathcal{D}}}

\newcommand{\mJ}{\ensuremath{\mathcal{J}}}

\newcommand{\mE}{\ensuremath{\mathcal{E}}}

\newtheorem{lem}{Lemma}[section]

\newtheorem{thm}{Theorem}

\newtheorem{cor}[lem]{Corollary}
\newtheorem{prop}[lem]{Proposition}
\newtheorem{defn}[lem]{Definition}

\newtheorem{rmk}[lem]{Remark}

\def\lto{\longrightarrow}
\def\lmto{\longmapsto}

\def\leq{\leqslant}
\def\geq{\geqslant}

\numberwithin{equation}{section}

\begin{document}

\title[Cauchy and uniform temporal functions]{Cauchy and uniform temporal functions of globally hyperbolic cone fields.}


\author{Patrick Bernard}
\address{Universit\'e  Paris-Dauphine,\\
PSL Research University,\\
\'Ecole Normale Sup\'erieure,\\
DMA (UMR CNRS 8553)\\
45, rue d'Ulm\\
75230 Paris Cedex 05,
France}
\curraddr{}
\email{patrick.bernard@ens.fr}
\thanks{}

\author{Stefan Suhr}
\address{Fakult\"at f\"ur Mathematik,\\
Ruhr-Universit\"at Bochum,\\
Universit\"atsstra\ss e 150\\
44780 Bochum,
Germany}
\curraddr{}
\email{stefan.suhr@rub.de}
\thanks{This research has been supported by the SFB/TRR 191 ``Symplectic Structures in Geometry, Algebra and Dynamics'', funded by the Deutsche Forschungsgemeinschaft.}


\date{}

\dedicatory{}

\commby{}

\begin{abstract}
We study a class of time functions called uniform temporal functions in the general context
 of globally hyperbolic closed cone fields. We prove some existence results for uniform temporal functions,
 and  establish  the density of uniform temporal functions in Cauchy causal functions.
\end{abstract}

\maketitle

\section{Introduction}
Let us first recall some definitions from \cite{BS1}, see also \cite{mi17}.
We say that $C\subset E$ is a {\it convex cone} in the vector space $E$ if
it is a (possibly empty)  convex set $C$ contained in an open half space  which is positively homogeneous:
$tx\in C$ if $t> 0$ and $x\in C$.
A  {\it cone field} on a smooth manifold $M$ is a subset $\mC\subset TM$ such that $\mC(x):= \mC\cap T_xM$ is a convex cone for each 
$x\in M$. We say that $\mC$ is a {\it closed cone field} if it is a cone field such that $\mC\cup T^0M$
is closed in $TM$, where $T^0M$ is the zero section of the tangent bundle. Our definitions  imply that cone fields
are disjoint from $T^0M$.
We say that $\mE$ is an {\it open cone field} if it is a cone field open in $TM$.
The main example of closed cone fields  we have in mind is the set of future pointing non spacelike vectors associated to
{ a continuous time} oriented Lorentzian metric. Our setting is more general, but our main results are new even in this 
important case.
The main example of open cone field is the set of future pointing timelike vectors associated to a continuous time
oriented Lorentzian metric.

The {\it domain of $\mC$} is the set of points $x$ such that $\mC(x)$ is not empty. It is a closed subset of $M$.
The points where $\mC(x)$ is empty are called {\it degenerate points}.

For a closed cone field $\mC$,  the curve  $\gamma:I\lto M$ is said to be  {\it $\mC$-causal}  (or just {\it causal}) if
it is locally Lipschitz and   if the inclusion $\dot \gamma(t) \in \mC(\gamma(t))$ holds for almost all $t\in I$. It is convenient to declare  that the trivial curve $\gamma:\{0\}\lmto x$ is causal if and only if $x$ belongs to
the domain of $\mC$.
The  causal  future  $\mJ_{\mC}^{+}(x)$ of $x$ is the set of points $y\in M$ such that there exists a 
  causal curve  $\gamma:[0,T]\lto M, T\geq 0$
satisfying $\gamma(0)=x$ and $\gamma(T)=y$.
The strictly causal  future  $\mJ_{\mC}^{++}(x)$ of $x$ is the set of points $y\in M$ such that there exists a 
non trivial  causal curve  $\gamma:[0,T]\lto M, T>0$
satisfying $\gamma(0)=x$ and $\gamma(T)=y$.
We thus have 
$\mJ^+_{\mC}(x)=\mJ_{\mC}^{++}(x)\cup \{x\}$ if $x\in \mD$, and $\mJ_{\mC}^{++}(x)\subset \mJ^+(x)$ are empty otherwise.
The causal past $\mJ_{\mC}^-(x)$ is the set of points $x'\in M$ such that $x\in \mJ^{+}_{\mC}(x')$.
More generally, for each subset $A\subset M$, we denote by $\mJ^{\pm}_{\mC}(A):= \cup_{x\in A}\mJ^{\pm}_{\mC} (x)$ the causal future and past
of $A$. 
We have the inclusion $\mJ_{\mC}^+(y)\subset \mJ_{\mC}^+(x)$ if 
$y \in \mJ_{\mC}^+(x)$.

The closed cone field $\mC$ is said to be \textit{locally solvable}  if, for each point $x$ in the domain of $\mC$,
there exists a  nonconstant causal curve $\gamma :]-1, 1[\lto M$ satisfying $\gamma(0)=x$
(if this property is satisfied for the point $x$, we say that $\mC$ is solvable at $x$). A closed cone field is 
always locally solvable in the interior of its domain. In particular, closed cone fields without
degenerate points are locally solvable.
We will say that the closed cone field $\mC$ has {\it smooth domain} if its domain $\mD$ is a 
submanifold with boundary. This case includes the  setting of \cite{Sa13}. Note that the hypothesis of timelike boundary  in \cite{Sa13} implies local solvability.

We say that the closed cone field  $\mC$ is  {\it causal} if $x\not\in \mJ^{++}(x)$ for each $x\in M$.
We say that $\mC$ is {\it hyperbolic} if it is causal and if moreover $\mJ^+(K)\cap \mJ^-(K')$ is compact for each compact sets  $K,K'\subset M$. We say that $\mC$ is {\it globally hyperbolic} if it  is hyperbolic and 
locally solvable (this definition is   more general than the one in \cite{BS1} as it 
allows the possibility of degenerate points).

In the case where $\mC$ contains a non-degenerate open cone field (for example in the case of a continuous time oriented Lorentzian metric), the above definition 
 is equivalent to the usual (since \cite{besa5}) condition that $\mC$ is  causal and that  the diamonds $\mJ^+(x) \cap \mJ^-(x')$ are all compact, as was proved in \cite{BS1}.

A { continuous} function $\tau: M\lto \Rm$ is said {\it causal} if $\tau \circ \gamma$ is non decreasing for each causal curve $\gamma$;
a {\it time function} if $\tau \circ \gamma$ is increasing for each causal curve $\gamma$; {\it temporal} if it is smooth and if  $d\tau_x\cdot v>0$ for each $(x,v)\in \mC$.

A smooth function is causal if and only $d\tau_x \cdot v \geq 0$ for each $(x,v)\in \mC$.
A temporal function is a time function, and a time function is causal.
Being causal, a time function, or temporal depends only on the values of $\tau$
in a neighborhood of the domain.

 A causal curve $\gamma\colon I\lto M$ is said {\it complete} if $I$ is an open interval and neither $\lim_{t\uparrow \sup I}\gamma(t)$ nor $\lim_{t\downarrow \inf I} \gamma(t)$ exist. A causal curve $\gamma\colon I \to M$, where $I$ is a 
 (not necessarily open) time interval,  is said {\it inextensible} if it cannot be extended to a larger time interval.
 Each causal curve can be extended to an inextensible causal curve. Complete causal curves are inextensible,
 but the converse is true only in the locally solvable case. More precisely :
 
 \begin{lem}
 	The closed cone field $\mC$ is locally solvable if and only if all inextensible causal curves are complete.
 \end{lem}
 
 \begin{proof}
 If $\mC$ is locally solvable, then a causal curve $\gamma :]a,b[\lto M$ which has a limit 
 at $b$ (resp. $a$) can be extended to a causal curve defined on $]a, b+1[$ (resp. $]a-1, b[$).
 Conversely, assume that each inextensible causal curve is complete. Then for each $x$ in the domain 
 of $\mC$, the trivial causal curve $\{0\}\lmto x$ can be extended to a complete curve defined on some open interval $]a,b[$. By reparametrization, we can assume that $a=-1, b=1$.
\end{proof}

The causal function $\tau$ is said {\it Cauchy} if $\tau \circ \gamma: I\lto \Rm$ is onto for each 
 inextensible causal curve $\gamma:I\lto M$. The existence of a Cauchy causal function implies that
 inextensible causal curves are complete, hence that $\mC$ is locally solvable.
 
It is known in the classical setting   that a spacetime is globally hyperbolic if and only if it admits
a Cauchy temporal function, as was proved in \cite{besa3}, see also  \cite{fasi, mi15, musa, mi17}. 
In order to study this question in the present setting, it is useful to consider the notion of 
uniform temporal functions, which was introduced in  \cite{BS1} 
where they are called steep temporal functions:

\begin{defn}
Given a { Riemannian} metric $g$, the temporal function $\tau$ is said $g$-uniform (or uniform with respect to $g$) if  $d\tau_x\cdot v \geq |v|^g_x := \sqrt{g_x(v,v)}$
for each $(x,v)\in \mC$.

The function $\tau$ is said completely uniform if it is uniform with respect to some complete metric.
\end{defn}

The motivation for this definition comes from :

\begin{lem}
	If $\mC$ is a locally solvable cone field, and $\tau$ a completely uniform temporal function for $\mC$, then $\tau$ is  a Cauchy temporal function for $\mC$.
\end{lem}

\begin{proof}
 Since inextensible curves have infinite length with respect to a complete metric and $|\tau\circ \gamma(t)-\tau\circ \gamma(s)|$ is bounded from the below 
by the length of $\gamma|_{[s,t]}$, the claim is immediate.
\end{proof}

We give an example in Section \ref{sec-ex} showing that the converse of this lemma does not hold: not all Cauchy temporal functions are completely uniform. 

We proved in \cite{BS1} that a  hyperbolic cone field admits uniform temporal functions with respect to any metric, hence completely uniform temporal functions.
As a consequence, globally hyperbolic cone fields admit Cauchy temporal functions.
Conversely, the existence of a Cauchy causal function implies global hyperbolicity, this follows from  
Proposition \ref{prop-ch} below.

Even more, the existence of a completely uniform temporal function  implies that the hyperbolic cone field $\mC$ admits a hyperbolic enlargement (a hyperbolic closed cone field $\mC'$
which contains $\mC$ in its interior).
 Indeed, if $g$ is a complete metric and $\tau$ is $g$-uniform, then 
 $\tau$ has no critical point on the domain $\mD$ of $\mC$. Then the function $\tau$ 
 is  $g/4$-uniform (hence completely uniform) for the enlargement 
 $$\mC':= \{(x,v):   d\tau_x\cdot v\geq  |v|^g_x/2, v\neq 0 \}.$$
 
This extends \cite{BM11} and \cite[Theorem 1.2]{fasi}. The cone field $\mC'$
constructed above is not necessarily locally solvable, but it can be made so by the following lemma provided that $\mC$ is locally solvable. This  proves that each  globally hyperbolic cone field has a  globally hyperbolic enlargement.

\begin{lem}\label{lem-solvable}
	Each locally solvable closed cone field $\mC$ admits a locally solvable enlargement.
	More precisely, if  $\mC_1$ is an enlargement of $\mC$, then there exists a closed neighborhood $F$ of $\mD$ such that 
  $\mC':=\mC_{1|F}$ is locally solvable. 
	
	In the case where $\mC$ has a smooth boundary, the closed neighborhood  $F$ can be assumed to have smooth boundary too.
\end{lem}
	
\begin{proof}
We endow $M$ with a complete Riemannian metric $g$. 
We define $F$ as the union of the images 
of all $\mC_1$-causal curves of lengh less than one and with endpoints in $\mD$. It is easy
to see that $F$ contains $\mD$ in its interior.
It follows from the limit curve lemma, see \cite[Section 2.3]{BS1} that $F$ is closed.
Let us prove that $\mC_{1|F}$ is locally solvable.
By the definition of $F$, each point $z\in F$ is of the form $z=\gamma(t)$ for a $\mC_1$-causal curve $\gamma:[-1,1]\lto M$ having its endpoints in $\mD$.
If $t\in ]-1,1[$, then the curve $\gamma$ solves $\mC_{1|F}$ at points $z$.
If $t\in \{-1,1\}$, then $z\in \mD$, and then $\mC$ is solvable at $z$, hence so is $\mC_{1|F}$.

We make a different construction in the case with smooth boundary.
There exist  smooth non singular vector fields $X^{\pm}$ which are defined in a neighborhood 
of $\mD$, are contained in $\mC_1$, and points strictly to the interior  (resp. the exterior) of $\mD$ at points of the boundary. 
Then, there exists a neighborhood $F$ of $\mD$ with smooth boundary such that $X^+$
still points inside and $X^-$ outside. This implies that $\mC_{1|F}$ is locally solvable. 
\end{proof}

It is not hard to prove (but not very useful) that each temporal function (not necessarily Cauchy) is uniform with respect to some (not necessarily complete) Riemannian metric. As we just saw, in order to be completely uniform a temporal function has to be Cauchy. We will however see in Section \ref{sec-ex} that a Cauchy temporal function is not necessarily completely uniform.

Our goal in the present paper is to clarify the differences between the notions of completely uniform 
temporal functions and of Cauchy temporal functions. In order to do so, we will prove some existence results of completely uniform temporal functions which mimic corresponding results for Cauchy temporal functions. We will deduce that each Cauchy temporal function can be uniformly approximated by 
completely uniform temporal functions.

The subset  $H\subset M$ is said to be {\it acausal} if   $H\cap \mJ^{++}(x)$ is empty for each $x\in H$;
{\it spacelike}  if there exists an open neighborhood $\tilde M$ of the domain $\mD$ such that 
$H\cap \tilde M$  is a smooth  hypersurface of $\tilde M$, and  if $T_xH$ is disjoint from $\mC(x)$ for each $x\in H$;
{\it spatial}  if it is  spacelike, intersects each inextensible causal curve at one
and only one point, and if moreover $\bar H \cap \mD=H\cap \mD$.

 Each level set of a Cauchy temporal function is spatial.
Our first results is a converse which   extends  \cite[Theorem 1.2]{besa4} to the more general setting of closed cone fields and to uniform temporal functions (instead of 
Cauchy temporal functions) :

\begin{thm}\label{thm-one}
	Let $H$ be a spatial subset and $g$ a Riemannian metric.
	Then there exists a $g$-uniform temporal function which is null on $H$.
\end{thm}

A proof is given in Section \ref{sec-one}, it is different from the one  in \cite{besa4}.
Note that Theorem \ref{thm-one} corresponds to case (B) of \cite[Theorem 1.2]{besa4}.
The other cases, when $H$ may be not spacelike or not acausal, are also extended in the present paper,
 see Corollary \ref{cor-cc} below.
 Theorem \ref{thm-one} implies the stability of spatial subsets:
 
 \begin{cor}\label{cor-ce}
 	If $H$ is a spatial subset for the cone field $\mC$, then there exists
 	a globally hyperbolic enlargement $\mC'$ of $\mC$ such that $H$ is  spatial for $\mC'$.
 \end{cor}
 
 The stability of spatial subsets implies the stability of Cauchy temporal functions: 
 
 \begin{cor}\label{cor-scf}
 	If $\tau$ is a Cauchy temporal function for $\mC$, then there exists a globally hyperbolic  enlargement $\mC'$ of $\mC$
 	such that $\tau$ is a Cauchy temporal function for $\mC'$.
 \end{cor}
 
 In the Lorenzian case, this result is \cite[Theorem 4.2]{Sa13}. This paper  considers the case with boundary, which is also included in our setting.
 In the setting of non-degenerate closed cone fields, it was obtained  in \cite[Theorem 2.46]{mi17}.
 
 \begin{proof}
 The levels $N_i:=\{\tau=i\}$ are spatial. 
 For each $i\in \Zm$, we consider a globally hyperbolic enlargement $\mC_i$ of $\mC$
 such that $N_i$ is spatial for $\mC_i$, and $\tau$ is temporal (but not necessarily Cauchy).
 We then consider the closed cone field $\mC'$ which is equal to $\mC_i\cap \mC_{i+1}$
 on  $\tau^{-1}(]i,i+1[)$ for each $i$, and to $\mC_i$ on $N_i$.
 It is not hard to verify that $\mC'$ is a globally hyperbolic cone field, and that all
 the levels $N_i$ are  spatial for $\mC'$, so that $\tau$ is a Cauchy temporal function. 
 \end{proof}

 We also  have a product structure if the domain is smooth (the result does not hold in general for globally hyperbolic cone fields with degenerate points):
 
 \begin{cor}\label{cor-prod}
 	Let $\mC$ be a globally hyperbolic cone field with smooth boundary,  let $V$ be a neighborhood of $\mD$, and let $\tau$ be a completely uniform  temporal function of $\mC$.
 	There exists a  neighborhood $U\subset V$ of $\mD$, and a smooth map $\pi: U\lto N=\{x\in U : \tau(x)=0\}$
 	such that $(\pi, \tau):U\lto N\times \Rm $ is a diffeomophism which sends $\mD$ to 
 	$(N\cap  \mD)\times \Rm$.
 	\end{cor}
 
 The proof will be given in section \ref{sec-smooth}.

It is not much harder to find a Cauchy temporal function with two given level sets
 than with one given level set (see Proposition \ref{prop-twocauchy} below) but the following statement is more delicate, and is the main result of the present paper:

\begin{thm}\label{thm-two-unif}
	Let $\mC$ be a globally hyperbolic cone field with smooth boundary.
	Let $H_0,H_1$ be two spatial subsets such that $H_1\cap  \mJ^-(H_0)=\emptyset$. Then there exists a completely uniform
	temporal function $\tau\colon M\to \RR$ such that  $\tau=i$ on $H_i$ for $i=0,1$ (hence 
	$\mD\cap \tau^{-1}(i)= \mD\cap H_i$).
\end{thm}

The proof in the case where $\mD=M$ is given in Section \ref{sec-two}, the full result is deduced in 
Section \ref{sec-smooth}.
We do not know if the assumption that the domain has smooth boundary is necessary.
 We deduce in Section \ref{sec-unif}:

\begin{thm}\label{thm-dense}
	Let $\mC$ be a globally hyperbolic cone field with smooth boundary, 
	let $\tau$ be a Cauchy causal function, and let $\epsilon>0$ be given.
	Then there exists a completely uniform temporal  function $v$ such that 
	$\sup |v-\tau|\leq\epsilon$.
\end{thm}

We will however see in Section \ref{sec-ex} an example of a Cauchy temporal function which is not completely uniform.

\section{Cauchy subsets of locally solvable closed fields}

We consider in this section a locally solvable closed cone field $\mC$.
\begin{defn}\label{def1}
We say that $H\subset M$ is a Cauchy subset if there exists a neighborhood $\tilde M$ of the domain $\mD$
and two open sets $U^{\pm}$ of $\tilde M$ such that $U^+, U^-$   and $\tilde M \cap H$ form a partiction of  $\tilde M$ and such that, for each inextensible causal curve $\gamma : \Rm \lto M$, there exist two finite times $a\leq b$ such that 
$\gamma(]-\infty, a[)\subset U^-$, $\gamma(]b,,\infty[)\subset U^+$ and
$\gamma([a,b])\subset H$.
\end{defn}
This slightly complicated definition is designed to be as coherent as possible with the  terminology used in the classical setting of Lorentzian manifolds.

\begin{prop}\label{prop-ch}
	If the closed cone field $\mC$ is locally solvable and admits a Cauchy subset, then it is globally hyperbolic.
\end{prop}

\begin{proof}
It follows immediately from the definition that inextensible causal curves start in $U^-$ and end in $U^+$,
hence they cannot be periodic: the cone field is causal.

We now prove the compactness of $\mJ^+(K_1)\cap \mJ^-(K_2)$ when $K_1$ and $K_2$ are compact.
If $x_n$ is a sequence in $\mJ^+(K_1)\cap \mJ^-(K_2)$ , then for each $n$ there exists 
an inextensible causal curve  $\gamma_n:\Rm \lto M$, parametrized by arclength, and  
either times $s_n\geq 0$, $L_n\geq t_n\geq 0$, such that $\gamma_n(0)\in K_1$, $\gamma_n(t_n)=x_n$, $\gamma_n(s_n)\in H$, $\gamma_n(L_n)\in K_2$,
or times   $s_n\leq 0$, $-L_n\leq t_n\leq 0$, such that $\gamma_n(0)\in K_2$, $\gamma_n(t_n)=x_n$, $\gamma_n(s_n)\in H$, $\gamma_n(-L_n)\in K_1$.

By an extraction, we can suppose that only one of the two situations occurs, we will treat the first case, the second being similar.

The curves $\gamma_n$ converge, uniformly on compact sets, to an inextensible causal curve $\gamma: \Rm\lto M$ (see \cite[Proposition 2.14]{BS1}).
Since $H$ is Cauchy, there exists $\theta$ such that $\gamma(\theta)\in U^+$, hence $\gamma_n(\theta)\in U^+$ for large $n$, 
hence $s_n\leq \theta$.  The sequence $s_n$ is thus  bounded.

 If $L_n \leq s_n$ for infinitely many values of $n$, then  up to extraction, $x_n=\gamma_n(t_n)\lto \gamma(t)\in A$.

Otherwise, we can assume that $L_n\geq s_n$ for each $n$. Then, we can prove as above, by considering 
a limit of the curves $\gamma_n(t-L_n)$, that $L_n-s_n$ is bounded and the conclusion follows.
\end{proof}

%

Another lemma will be useful :

\begin{lem}\label{lem-F}
	Let $\mC$ be a globally hyperbolic cone field, and $F\subset M$ be a closed past subset, i.e. $\mJ^-(F)\subset F$.
	Then there exists a smooth causal function $ v$ which is null on $F$, positive and  temporal outside of $F$,
	and such  that $ v \circ \gamma(\Rm)\supset ]0, \infty[$ for each inextensible causal curve which is not contained in $F$. If $H$ is acausal, then $v$ is a smooth time function.
\end{lem}

\begin{proof}
Let $U$ be the complement of $F$. It is easy to see that  $(U,\mC|_U)$ is globally hyperbolic, hence 
admits a smooth Cauchy temporal function $\tau:U\lto \Rm$.
Let $\eta_k\colon \Rm\lto [0,\infty)$ be a smooth function such that $\eta_k(t)=0$
for $t\leq -k$, $\eta'_k(t)>0$ for $t>-k$, and $\eta_k(t)=t$ for $t\geq 1$. Each of the functions 
$\eta_k\circ \tau$ extends by $0$ to a smooth function $v_k$ on $M$. 
If $\{a_k\}_{k\in \mathbb{N}}$ is a sequence of positive numbers which decreases sufficiently fast,
(see \cite[Lemma 3.2]{Fa} for the existence of such a sequence) then $\eta :=\sum a_k \eta_k$ is a smooth diffeomorphism from
$\Rm$ to $]0, \infty[$ and  $v:=\sum a_kv_k$ is a smooth function on $M$, which is null on $F$, and which is 
positive and temporal  on its complement. Finally, if $\gamma$ is an inextensible causal curve for $\mC$
which is not contained in $F$, then either $\gamma(\Rm)$ is contained in $U$, and then $\gamma$
is an inextensible causal curve of $(U,\mC|_U)$, hence $\tau\circ \gamma(\Rm)=]-\infty, \infty[$
and $\tau\circ \gamma(\Rm)=]0, \infty[$; or $\gamma$ enters $F$ and then $v\circ \gamma (\Rm)=[0,\infty[$. 
\end{proof}

%
%
%
%

\begin{prop}\label{prop-far}
	Let $H\subset M$ be a Cauchy subset, $U$ an open neighborhood of $H$, and $g$ a Riemannian metric. Then there exists a smooth causal Cauchy  function
	$v \colon M\to \Rm$ satisfying   $v|_H \equiv 0$ and $v^{-1}(0)=H$ in a neighborhood of $\mD$ , which is temporal outside of $H$ and $g$-uniform outside of $U$.  In the case where $H$ is moreover  acausal, $v$ is a smooth Cauchy time function.
\end{prop}

\begin{rmk}\label{cor-cc}
For $U=M$ the proposition extends \cite[Theorem 5.15]{besa4}.
\end{rmk}

\begin{proof}[Proof of Proposition \ref{prop-far}.]
We denote by  $J^\pm:=(H\cap \tilde M)\cup U^{\pm}$, where $\tilde M$ and $U^{\pm}$
are  chosen according to Definition \ref{def1}.

For each compact set $K\subset J^+$, the set $\mJ^-(K) \cap J^+=\mJ^-(K)\cap \mJ^+(H)$ 
is compact. We can therefore build a sequence of compact sets $K_n\subset J^+$
such that $J^+=\cup _n K_n$,  
$\mJ^-(K_{n})\cap J^+\subset K_n$, and such that $K_{n-1}$  is contained  in  the interior $\mathring K_n$ of $K_{n}$ in  $J^+$  for each $n\geq 1$.

Using Lemma \ref{lem-F} in $\tilde M$ with $F=J^-$, we find a smooth causal function 
$v_0: \tilde M\lto [0, \infty[$ which is null on $J^-$, positive and temporal on the complement of $F$, and which moreover has the property that $v_0\circ \gamma$
takes all values in $[0, \infty)$ when $\gamma$ is an inextensible causal curve.

For each $n\geq 1$, we set $F_n=K_n \cup J^-$. By Lemma \ref{lem-F}, there exists a smooth 
causal function $v_n$ on $\tilde M$ which is  null on $F_n$
and which is positive and temporal outside of this set. 

Since the set $K_2\setminus U$ is compact and disjoint from $J^-$, there exists a positive constant $a_0$ such that $a_0 v_0$ is uniform on $K_2\setminus U$.
Then for each $n\geq 1$, 
since the set $F_{n+2}\setminus (U\cup \mathring K_{n+1})$ is compact and disjoint 
from $F_n$, there exists a 
positive number $a_n$ such that $a_nv_n$ is uniform  on $F_{n+2}\setminus (\mathring K_{n+1}\cup  U)$.
The locally finite sum
$v^+:= \sum_{n\geq 0} a_n v_n$
is a smooth causal  function which is null on $ J^-$, temporal outside of this set, and 
uniform on $J^+\setminus U$. Moreover, $v^+\circ \gamma$ takes all positive values if $\gamma$ is an inextensible causal curve.

 In a similar way we prove  the existence of a smooth causal function $v^-$ on $\tilde M$ which is null on $J^+$,
negative and temporal outside of this set,  uniform on $\mJ^-(H)\setminus U$, and such that $v^-\circ \gamma$ takes
all negative values for each inextensible causal curve $\gamma$.
The function $\tilde v:=v^++v^-$ is a smooth causal Cauchy function  on $\tilde M$ such that
 $\{\tilde v=0\}=H$.

In order to obtain the desired function $v$ on the whole manifold $M$, we just consider a function $h$ which
is equal to $1$ in a neighborhood of $\mD$ and null outside of $\tilde M$, and take $v=h\tilde v$. 
\end{proof}

\section{One spatial hypersurface}\label{sec-one}

We prove Theorem \ref{thm-one} and then give a proposition on the extension of  acausal compact sets 
to spatial hypersurfaces.

\begin{proof}[Proof of Theorem \ref{thm-one}.]
We first prove the existence of a smooth causal function $u$ which is null on $H$ and
$g$-uniform in a neighborhood $U$ of $H$, see Proposition \ref{prop-near} below.

Since $H$ is a Cauchy subset (see Corollary \ref{cor-sc} below), we can apply Proposition \ref{prop-far}, and get the existence of a smooth causal function $v$  which is null on $H$ and
uniform outside of $U$. We deduce that $\tau:= u+v$ is a uniform temporal function null on $H$.
\end{proof}

The first step is a variant of \cite[Proposition 5.1]{BS1}.

\begin{prop}\label{prop-near}
	Let $H\subset M$ be a spatial subset, $g$ be a Riemannian metric, and $V$ be a neighborhood of $H$. Given $\alpha\in ]0,1[$, there exists a  smooth causal function $u: M\lto [-1,1]$ with the following properties:
	\begin{itemize}
		\item
	$u_{|H}=0$.
	\item
	The function $u$ is $g$-uniform on $u^{-1}(]-\alpha,\alpha[)$ and temporal on  $u^{-1}(]-1,1[)$.
	\item
	The function $u$ is equal to $+1$ in a neighborhood of $\mJ^+(H)\setminus V$ and to  $-1$ in a neighborhood of $\mJ^-(H)\setminus V$.
	\end{itemize}
\end{prop}

\begin{proof}
We consider a spatial hypersurface (of $M$)  $\hat H\subset H$ which is smooth and contains $H\cap \mD$, 
and we work first in a tubular neighborhood $\hat M\approx\hat  H\times \Rm\subset V\subset M$, 
such that $(\hat H \times [-1,1]) \cap \mD$ is closed in $M$.  

There exists a smooth positive function $y\mapsto \delta(y)$ on $\hat H$ such that 
$$ \mC(y,0)\subset \{(v_y,v_z): v_z\geq 3\delta(y) \|v_y\|, v_z \geq 3\delta (y) \|(v_y,v_z)\|\}
$$
for each $y\in\hat  H$. Here we have used the notation $\|v_y\|:= \|(v_y,0)\|_{(y,z)}$, where the norms are measured with respect to the metric $g$ on $M$.
Then, there exists a smooth positive function $\epsilon : H\lto ]0,1[$ such that
$$ \mC(y,z)\subset \{(v_y,v_z): v_z\geq 2\delta(y) \|v_y\|, v_z \geq 2\delta (y) \|(v_y,v_z)\|\}
$$
provided $|z|\leq \epsilon(y)$. 
We  also assume that $\epsilon \leq \delta$.

Let $f:\hat H\lto \Rm$ be a smooth positive function such that 
$\|df_y\|+f(y)\leq \epsilon(y)$ for all {$y\in  \hat H$},
(see for example  \cite[Lemma 5.4]{BS1} for the existence of such a function).
We set 
$$\hat u(y,z)=\phi(z/f(y)),
$$
where $\phi:\Rm\lto[-1,1]$ is a smooth nondecreasing function which has positive derivative on $]-1,1[$
and is equal to $1$ on $[1,\infty)$ and to $-1$ on $(-\infty,-1]$. We moreover assume that 
$\phi(t)=t$ on $[-\alpha,\alpha]$.

On the open set 
$ \Omega:=\{(y,z): |z|<f(y)\}=\hat u ^{-1}(]-1,1[)$, we have
$$d\hat u_{(y,z)}\cdot(v_y,v_z)=\frac{\phi'(z/f(y))}{f(y)}
\big(
v_z-\frac{z}{f(y)}df_y\cdot v_y
\big)\geq \frac{\phi'(z/f(y))}{2f(y)}v_z
$$ 
for $(v_y,v_z)\in  \mC(y,z)$ since
$
|(z/f(y))df_y\cdot v_y|\leq \delta(y) \|v_y\|\leq v_z/2.
$
This implies that $\hat u$ is temporal on  $\Omega$ .
Outside of this open set, $\hat u$ is locally constant, hence it is a causal function.

Finally, on $\hat u^{-1}(]-\alpha,\alpha[)$, we have $\phi'(z/f(y))=1$ hence
$$d \hat u_{(y,z)}\cdot (v_y,v_z)\geq  \frac{1}{2f(y)} v_z\geq \frac{\delta(y)}{f(y)} \|(v_y,v_z)\|_{(y,z)}
\geq  \|(v_y,v_z)\|_{(y,z)}$$
for each $(v_y,v_z)\in \mC(y,z)$. In other words, the causal  function  $\hat u$ is uniform on 
$\hat u^{-1}(]-\alpha,\alpha[)$.

Let us now extend the function $\hat u$ from $\hat M$ to $M$. The domain $\mD$ of $\mC$ is the union of 
the closed sets $\mJ^{\pm}(H)$, and the intersection of these sets is $H\cap \mD$. Let $F^{\pm}$ be the disjoint closed sets  $\mJ^{\pm}(H) \setminus \hat M$. Let $W^{\pm}$ be  open sets containing $F^{\pm}$, 
which are disjoint from each other, and also disjoint from  
$(\hat H\times[-1,1])\cap \mD\subset \hat M$ and from the closure $\bar H$ of $H$ (this is possible since $\bar H \cap \mD= H\cap \mD$). 
We consider a partition  of the unity  $(a,b,c,d)$ associated to the finite open covering 
$M=\hat M\cup  W^+\cup W^-\cup(M\setminus \mD)$ and set
$u:= a\hat u +b-c$.
\end{proof}

This proposition has the following  corollary :

\begin{cor}\label{cor-sc}
	Each spatial subset  is a Cauchy subset.
\end{cor}

\begin{proof}
We consider the function $u$ obtained in the  Proposition \ref{prop-near}.
There exists a neighborhood $\tilde M$ of $\mD$ such that $\tilde M\cap u^{-1}(0)=\tilde M\cap H$.
Indeed, let $x_n$ be a sequence of points converging to $x\in \mD$ and such that $u(x_n)=0$.
Since $u(x)=0$ and $x\in \mD$, we have $x\in H$.
In the neighborhood of $x$, the function $u$ is non-degenerate, hence, locally,   $u^{-1}(0)=H$. This implies that 
$x_n\in H$ for large $n$, proving the claim.

We define $U^+:= \{x\in \tilde M : u(x)>0\}$, $U^-:= \{x\in \tilde M : u(x)<0\}$. They satisfy all the requirements of the definition of a Cauchy subset.
\end{proof}

Theorem \ref{thm-one} is proved.
Let us also mention :

\begin{lem}\label{lem-mod}
	Let $\tau$ be a Cauchy temporal function, and let $a\in \RR$. Let $g$ be a Riemannian metric.
	Then there exists a Cauchy temporal function $\tau_1$ such that $\tau_1=\tau$
	on $\{\tau\geq a+1/3\} \cup \{\tau=a\}\cup \{\tau \leq a-1/3\}$, and which is 
	$g$-uniform on $\tau_1^{-1}(]a-1/10, a+1/10[)$.
\end{lem}

\begin{proof}
We first apply Proposition  \ref{prop-near} in the hyperbolic strip $\tau^{-1}(]a-1/3,a+1/3[)$ to the spatial set  $H$ and the metric $16g$. We get a smooth
causal function $u: M\lto [-1,1]$ which is 
$16g$-uniform on $u^{-1} (]a-3/4,a+ 3/4[)$, which is equal to $0$ on $H$, and which is equal to $1$ on $\{\tau \geq a+1/3\}$ and to $-1$ on 
$\{\tau \leq a-1/3\}$. 

Second, we consider a smooth increasing diffeomorphism  $\eta: \Rm\lto \Rm$  such that
$\eta(a)=a$,  $\eta (t)= t-1/4$ on $[a+1/3, \infty)$ and $\eta(t)=t+1/4$ on
$(-\infty, a-1/3]$.  

Finally we set $\tau_1=u/4+\eta \circ \tau$. This function is equal to $a$ on $\{\tau= a\}$,
it is equal to $\tau$ on  $\{\tau\geq a+1/3\}\cup \{\tau \leq a-1/3\}$
it is $g$-uniform on
$\tau_1^{-1}(]a-1/10, a+1/10[)\subset u^{-1}(]a-3/4,a+3/4[)$
because $u/4$ is.
\end{proof}

It is not much more difficult to find a Cauchy temporal function with prescribed values on 
two spatial hypersurfaces than on one spatial  hypersurface. More precisely:

\begin{prop}\label{prop-twocauchy}
	Let $H_0,H_1$ be two spatial subsets such that  $H_1\cap\mJ^-(H_0)=\emptyset$, and let $g$ be a complete Riemannian metric on $M$.
	Then there exists a Cauchy temporal function $\tau$ such that  $\tau=i$ on $H_i$ for $i=0,1$,
	and which is $g$-uniform in $\{\tau\geq 9/10\}$ and in $\{\tau\leq 1/10\}$.
\end{prop}

In view of this result, the remaining difficulty in the proof of Theorem \ref{thm-two-unif}
is to interpolate between $H_0$ and $H_1$ in a uniform way. This will be done in Section 
\ref{sec-two}.

\begin{proof}
The open set $M\setminus \mJ^-(H_0)$ is globally hyperbolic and contains $H_1$.
As a consequence, there exists a $4g$-uniform temporal function 
$\tilde \tau_1:M\setminus \mJ^-(H_0)\lto \Rm$ which is equal to $1$ on $H_1$. 
In view of (the proof of)  Lemma \ref{lem-F}, there exists a smooth  function
$\eta: \Rm \lto ]0, \infty)$ with positive derivative such that $\eta (t)=t$ for $t\geq  1/2$ and such that  $\eta\circ \tilde \tau_1$ extends (with the value $0$) to a smooth function on $M$,
that will be denoted by $\tau_1$.

By a similar method in the globally hyperbolic manifold $\{\tau_1<3/4\}$, there 
exists a smooth function $\tau_2: M\lto (-\infty, 1]$ which is equal to 
$1$ on $\{\tau_1\geq 3/4\}$, temporal in  $\{\tau_1<3/4\}$, null on $H_0$,  and $4g$-uniform in 
$\{\tau_2 \leq 1/2\}$.

The function  $\tau:= (\tau_1+\tau_2)/2$ is temporal, it is $g$-uniform
in $\{\tau \geq 9/10\}\subset \{\tau_1\geq 4/5\}$ because $\tau_1/2$ is $g$-uniform on this set,
and it is $g$-uniform on $\{\tau\leq 1/10\}\subset \{\tau_2\leq 1/5\}$ because $\tau_2/2$
is $g$-uniform on this set.
\end{proof}

We finish this section with an extension result which is a  generalization of \cite[Theorem 1.1]{besa4}.

\begin{prop}\label{prop-ext}
	Let $\mC$ be a globally hyperbolic cone field without degenerate points, and let 
	$\Sigma$ be a  spacelike hypersurface. 
	Then each compact subset $K\subset \Sigma$
	is contained in a spatial hypersurface  $H$.
	
	Moreover, if $N^-, N^+$ are two spatial sets such that
	$\Sigma \subset  \mJ^+(N^- )\cap \mJ^-(N^+)$, the  spatial 
	hypersurface $H$ containing $K$ can be chosen inside 
	$
	\mJ^+(N^-)\cap \mJ^-(N^+)
	$
	and equal to $N^-$ at infinity.
\end{prop}

\begin{proof}
We prove the second part of the Proposition, the first part  easily follows from it.

By Corollary \ref{cor-ce}, there exist  globally hyperbolic enlargements $\mC^{\pm}$ of $\mC$ 
such that $N^{\pm}$ is a Cauchy hypersurface for $\mC^{\pm}$.

We associate to $K$ the transverse open cone field $\mE_{K}$, which is defined by:

$\mE_{K}(x)=T_xM$ if $x\not \in K$ and 

$\mE_{K}(x)$ is the component of  $T_X M\setminus T_x \Sigma$ which contains $\mC(x)$
if $x\in K$.

By definition, this open cone field contains $\mC$, hence there exists an enlargement $\mC'$ of $\mC$ which is contained in $\mE_K$ (by \cite[Lemma 2.8]{BS1}).
The intersection $\mC_1:= \mC^+\cap \mC^- \cap \mC'$ is a globally hyperbolic closed cone field, 
$N^{\pm}$ are spatial hypersurfaces of $\mC_1$, 
and $\Sigma$ is spacelike for $\mC_1$ at each point of $K$.

Note that the set of points of $\Sigma$ at which $\Sigma$ is spacelike is relatively open in $\Sigma$
and contains $K$.
We can thus  assume (by possibly reducing it), that  $\Sigma$ is a  spacelike hypersurface for $\mC_1$.

We  consider a Cauchy temporal function $\tau$ for $\mC_1$  such that $\tau=0$ on $N^-$ and $\tau=1$ on $N^+$, according to Proposition \ref{prop-twocauchy}.

We can finally assume, by possibly reducing $\mC_1$, that $K$ is acausal for $\mC_1$, and even that there
exists a compact subset $Z\subset \Sigma$, which contains $K$ in its relative interior, and which is 
acausal for $\mC_1$.

To prove this last claim, we consider a decreasing sequence $Z_n$ of compact subsets of 
$\Sigma$ each of which is contained in the relative interior of the previous one, and  
such that $K=\cap Z_n$. We also consider a decreasing sequence $\mC_n$ of enlargements 
of $\mC$, each of which is an enlargement of the next one, and such that $\mC=\cap \mC_n$
(the existence of such a sequence of enlargements is proved in \cite[Lemma 2.9]{BS1}).
We assume by contradiction that 
for each $n$ there exists a $\mC_n$-causal curve $\gamma_n:I_n\lto M$, $I_n \subset[0,1]$ starting and ending in $Z_n$ and parametrized by $\tau$, \textit{i.e.} such that $\tau\circ \gamma_n(t)=t$. Then at the limit (along a subsequence), we obtain a $\mC$-causal curve $\gamma:I\lto M$ 
starting and ending in $K$. The acausality of $K$ implies that this curve is reduced to a point $\kappa \in K$. A local study at $\kappa$ leads to a contradiction.

The cone field $\mC_1$ contains a non-degenerate open cone, hence a smooth vector field $X$.
Each inextensible  orbit of $X$ intersects $N^-$ exactly once.
This defines a smooth map $\pi:M\lto N^-$ which, to each point $x$, associates the intersection of its orbit with $N^-$. 
The map 
$$
M \ni x\lmto (\pi(x),\tau(x))\in N^- \times \Rm
$$
is a smooth diffeomorphism. In other words, we can assume that 
$M=N^-\times \Rm$, $N^-=N^-\times \{0\}$, $N^+=N^-\times \{1\}$,  that $\mC_1(x)$ contains the vector $(0,1)$ at each point, that the vertical coordinate $\tau$ is a Cauchy temporal function for $\mC_1$,  and that $\tau(\Sigma)\in [0,1]$.

The set $F:= \mJ_{\mC_1}^-(Z)\cup \{\tau\leq 0\}$ is a past set for $\mC_1$, i.e. $F=\mJ^-_{\mC_1}(F)$, and it is equal  to $\{\tau\leq 0\}$ at infinity (meaning, outside of a compact set of $M$). 
Its complement is therefore a trapping domain for $\mC$ in the terminology of 
\cite{BS1}. Since $Z$ is acausal, it is contained in the boundary of $F$.
In the identification $M=N^-\times \Rm$, the boundary of $F$ is the graph of a compactly 
supported  Lipschitz function $\tilde f: N^-\lto [0, 1]$, which is smooth in a neighborhood
of $\pi(K)$.
We can then regularize $\tilde f$ to a compactly supported smooth function
$f:N^-\lto [0, 1]$ which 
is equal to $\tilde f$ near $\pi(K)$, and the graph $H$ of which is spacelike for $\mC_1$.
We can use for example \cite[Proposition 4.5]{BS1} or a more standard smoothing procedure.
The graph $H$ is the desired spatial hypersurface.
\end{proof}

\section{A Cauchy temporal function which is not completely uniform}\label{sec-ex}

Let us consider the plane $\Rm^2$ endowed with the constant cone field
$$\mC(x)=\{(v_x,v_y): v_y\geq |v_x|\},$$ 
i.e. the Minkowsky plane.
In this plane, we consider the strip $U=\Rm\times ]-1,1[$. The restriction of $\mC$
to $U$ is globally hyperbolic.

Let $\varphi:]-1,1[\lto \Rm$ be a smooth  increasing function onto $\Rm$.
Then the function $\varphi\circ y$ is a Cauchy time function on $U$.

An easy computation shows that the function $\tau(x,y):= y/(1+x^2)$
is a temporal function on $U$. Indeed
$$
|\partial_x \tau|<\frac{2|x|}{(1+x^2)^2}\leq \frac{1}{1+x^2}=\partial_y \tau.
$$

The function $h(x,y)=\tau(x,y)+\varphi(y)$ is thus a Cauchy temporal function. 

\begin{prop}
	If $\varphi'(0)=0$, then $h$ is not completely uniform.
\end{prop}

\begin{proof}
We prove that, if  $g$ is a metric  such that $h$ is $g$-uniform,
then $g$ is not complete.

We consider the constant  $\mC$-causal vector fields $V^+=(1,1)$ and $V^-=(-1,1)$.
For the metric $g$, we have
$$
|V^+|_{(x,0)}+|V^-|_{(x,0)}\leq dh_{(x,0)} \cdot (V^++V^-)=2\partial_y h(x,0)=\frac{2}{1+x^2}.
$$ 
The $g$-length $\ell$ of the axis $y=0$  thus satisfies
$$
\ell=\int_{\Rm} \frac{1}{2}|V^+-V^-|_{(x,0)} dx\leq
\int _{\Rm} \frac{|V^+|_{(x,0)}+|V^-|_{(x,0)}}{2} dx\leq
 \int \frac{1}{1+x^2}dx
<\infty .
$$
This implies that the metric $g$ on $U$ is not complete. 
\end{proof}

\section{Two Cauchy Hypersurfaces}\label{sec-two}

We consider in this section a non degenerate ($\mD=M$) globally hyperbolic cone field $\mC$
and prove  Theorem \ref{thm-two-unif} under this assumption, i.e. we prove:

\begin{prop}\label{prop-two}
	Let $\mC$ be a globally hyperbolic non degenerate cone field with smooth boundary.
	Let $H_0,H_1$ be two spatial subsets such that $H_1\cap  \mJ^-(H_0)=\emptyset$. Then there exists a completely uniform
	temporal function $\tau\colon M\to \RR$ such that  $\tau=i$ on $H_i$ for $i=0,1$ (hence 
	$\mD\cap \tau^{-1}(i)= \mD\cap H_i$).
\end{prop}

It is convenient in this section to fix once and for all a complete Riemannian 
metric $g$ on $M$ (but we shall be led to consider also other metrics).
In view of Proposition \ref{prop-twocauchy}, Proposition \ref{prop-two} follows from:

\begin{prop}\label{prop-two-unif}
	Let $\tau$ be a smooth Cauchy temporal function, and $g$ be a complete Riemannian metric on $M$.
We assume that there exists  $a<b$ in $]0,1[$ such that $\tau$ is $g$-uniform on 
	$\tau^{-1}(]0,a[\cup ]b,1[)$.
	Then there exists a smooth Cauchy temporal function $\theta$ and a complete Riemannian metric $h$ on $M$ such that:
	
	The function  $\theta$ is $h$-uniform on $\tau^{-1}(]0,1[)$ and equal to $\tau$ outside this strip.
	
	The metric $h$ is equal to $g$ outside of the strip $\tau^{-1}(]0,1[)$, and it satisfies
	$$
	d_h(\tau^{-1}(0), \tau^{-1}(1))\geq 1/3.
	$$
\end{prop}

\begin{proof}
We assume without loss of generality that $a<1/10$, 
$b>9/10$. We introduce intermediate values  
$a_1\in ]0,a[$ and $b_1\in ]b,1[$.

We will construct a spatial hypersurface $G$ which oscillates between the hypersurfaces $\{\tau=a_1\}$
and $\{\tau=b_1\}$: 

\begin{lem}\label{lem-osc}
	There exists a spatial hypersurface $G\subset \tau^{-1}([a_1,b_1])$ such that :
	
	The intersection  $\mJ^{++}(G)\cap \{\tau < b_1\}$
	is the disjoint union of a countable family of bounded sets  $F_n, n\geq 0$ such that 
	$d_g(F_n, F_m)\geq 1$ if $n \neq m$.
	
	The intersection  $\mJ^{--}(G)\cap \{\tau >  a_1\}$
	is the disjoint union of a countable family of bounded sets  $P_n, n\geq 0$ such that 
	$d_g(P_n, P_m)\geq 1$ if $n \neq m$.
\end{lem}

Assuming the lemma for the moment, we finish the proof of  Proposition \ref{prop-two-unif}.

We consider a smooth modification $\eta\circ \tau$ of $\tau$,
where $\eta: \Rm\lto \Rm$ is an increasing diffeomorphism satisfying $ \eta(t)=t-1/3$  for $t>9/10$ and  $\eta(t)=t+1/3$ for $t<1/10$.

We apply Proposition \ref{prop-near}  to the hypersurface $G$ in the globally hyperbolic open set
$\tau^{-1}(]0,1[)$ and obtain a smooth causal function $u:M\lto [-1/6,1/6]$ which is equal to $1/6$ in
 $\{\tau \geq 1\} $, to $-1/6$ in $\{\tau \leq 0\}$, and which is  $g$-uniform on $u^{-1}(]-1/7, 1/7[)$.
We set 
$$V := \tau^{-1}((-\infty, a_1[)\cup u^{-1}(]-1/10, 1/10[) \cup \tau^{-1}(]b_1,\infty)),
$$
the temporal function $\eta \circ \tau +u$ is $g$-uniform  in a neighborhood
of $\bar V$, hence there exists a smooth function $\lambda: M\lto]0,1]$, which is equal to one on 
$V$,
and  such that $\eta \circ \tau +u$ is $\lambda g$-uniform for $\mC$.
We don't know, though, whether the metric $\lambda g$ is complete. 
To solve this problem, we define 
$$
\theta:= \eta\circ \tau + 2u
\quad, \quad 
h:= \lambda g + du\otimes du.
$$
The function  $\theta$ is the sum of the function $\eta\circ \tau +u$, which is $\lambda g$-uniform,
and $u$, which is $du\otimes du$-uniform. It is therefore   $h$-uniform. It is moreover equal 
to $\tau$ outside of the strip $\tau^{-1}(]0,1[)$.

We now prove  that the metric $h$ is complete. This implies that $\theta$ is completely uniform for $\mC$.
 We denote by $A$ the complement of $V$, and set $A^+:= A\cap \mJ^{+}(G)$ and $A^-:=A\cap  \mJ^{-}(G)$.
We first observe that $d_h(\{u\geq 1/10\}, \{u\leq -1/10\})\geq 1/5$.
Indeed, the $h$-length $L$ of a curve $\gamma$  is bounded below by $|u(\gamma(1))-u(\gamma(0))|$
and $u_{|A^+}\geq 1/10$, $u_{|A^-}\leq -1/10$.
This remark also proves that $d_h(\{\tau=0\}, \{\tau=1\})\geq 1/3$.

The closed set $A^+$ is the disjoint union of the  bounded sets $A_n^+:= A^+\cap F_n$,
and the closed set $A^-$ is the disjoint union of the  bounded  sets $A_n^-:= A^-\cap P_n$.
Since $A^+\subset \{\tau \geq 1/10\}$ and  $A^-\subset \{\tau \leq -1/10\}$, we have
$
d_h(A^+, A^-)\geq 1/5.
$

Let $\gamma$ be a curve connecting $A_n^+$ and $\cup_{n\neq m} A_m^+$. If $\gamma$  enters $A^-$, then 
$L_{h}(\gamma) \geq d_h(\{u\geq 1/10\}, \{u\leq -1/10\})\geq 1/5$.
If $\gamma$  does not enter $A^-$, then $L_{\lambda g}(\gamma)\geq L_{\lambda g}(\gamma\cap V) =L_{g}(\gamma \cap V)
\geq d_g(A^+_n,\cup_{n\neq m} A_m^+) \geq 1$, since $\lambda = 1$ on $V$. 
In both cases, $L_h(\gamma)\geq 1/5$, hence  (using similar arguments for $A_n^-$)
$d_h(A_n^+,\cup _{m\neq n} A_m^+)\geq 1/5$ and $d_h(A_n^-,\cup _{m\neq n} A_m^-)\geq 1/5$.
In view of these inequalities, a curve $\gamma$ of finite $h$-length $L$ visits only finitely many of the sets $A_i^{\pm}$. If $K$ is the union of these finitely many visited sets, then $K$ is compact,
and $\gamma$ is contained in the compact set $\{x: d_g (K,x)\leq L\}$ (recall that $g$ is a 
complete metric).
\end{proof}

\textsc{Proof of Lemma  \ref{lem-osc}:}
The proof consists of repeated applications of Proposition \ref{prop-ext}.
We denote by $N^+$ the hypersurface $\{\tau=b_1\}$, and by $N^-$ the hypersurface $\{\tau=a_1\}$. 
Further we  denote  by $S$ the open strip $\tau^{-1}(]a_1,b_1[)$
and by $\bar S$ its closure  $\tau^{-1}([a_1,b_1])$.  We set $B(F,r):=\{y: d(y,F)\leq r\}$.

By Proposition \ref{prop-ext}, applied to $-\mC$, there exists a spatial hypersurface $G_0^+\subset \bar S$ which
contains a point $x_0\in N^-$ and is equal to $N^+$ at infinity.
We set $F_0:=\mJ^{++}(G^+_0)\cap S$.

Then, by Proposition \ref{prop-ext}, there exists a spatial hypersurface 
$G^-_0\subset \bar S$ which contains $G^+_0\cap \mJ^+(B(F_0,1))$ and is equal to $N^-$ at infinity.
Note then that 
$$\mJ^{++}(G^-_0)\cap S\cap B(F_0,1)= F_0. $$
Indeed, we have $\mJ^-(G_0^+)\cap B(F_0,1)\subset \mJ^-(G_0^-)$ hence 
$\mJ^-(G_0^+)\subset \mJ^-(G_0^-) \cup B^c(F_0,1)$.
Taking the complements in $S$, we get the inclusion 
$\mJ^{++}(G^-_0)\cap S\cap B(F_0,1)\subset F_0. $ The opposite inclusion is obvious.
We set $P_0:= \mJ^{--}(G^-_0)\cap S$.

We construct inductively sequences $G^{\pm}_n$ of spatial hypersurfaces,
and  increasing sequences $\mF_n, \mP_n$ of bounded subsets of $S$  such that:
 (1) $G^{\pm}_n$ is  equal to $N^{\pm}$ at infinity for each $n$;
(2) $\mF_n =S\cap \mJ^{++}(G_n^+), \quad \mP_n =S\cap \mJ^{--}(G_n^-)$;
(3) $G^+_{n+1}$ contains $G^-_n \cap \mJ^-(B(\mF_n\cup \mP_n,1)\cup B(x_0, 2n))$;
(4) $G^-_{n+1}$ contains $G^+_{n+1} \cap \mJ^+(B(\mF_{n+1}\cup \mP_{n},1 )\cup B(x_0,2n+2))$.

We denote $P_{n+1}:= \mP_{n+1}\setminus \mP_n$ and $F_{n+1}:=\mF_{n+1}\setminus \mF_n$.

Note that at each step we have $G_{n+1}^+\subset \mJ^+(G_n^-)\cap \mJ^-(G_n^+)$ and
$G_{n+1}^-\subset \mJ^-(G_{n+1}^+){\tiny } \cap \mJ^+(G_n^-)$.

Denoting $B_n:=  B(\mF_n\cup \mP_n,1)\cup B(x_0, 2n)$ and $B_n^c$ its complement,
(3) implies that that $\mJ^+(G_n^-)\cap B_n\subset \mJ^+(G_{n+1}^+)$.
This implies that $\mJ^+(G_n^-)\subset \mJ^+(G_{n+1}^+)\cup B_n^c$.
Taking the complements and intersecting with $S$ yields 
$$\mJ^{--}(G_{n+1}^+)\cap B_n \cap S\subset \mJ^{--}(G_n^-)\cap S=\mP_n$$
hence $\mP_{n+1}\cap B_n=\mP_n$,
and in particular $d(P_{n+1}, \mP_n)\geq 1$, hence $d(P_n,P_m)\geq 1$ 
for $n\neq m$.
Similarly, we have
$\mF_{n+1}\cap \big(B(\mF_{n}\cup \mP_{n-1},1)\cup B(x_0, 2n+2)\big)=\mF_n$.
Finally, we observe that $G_{n+1}^+=G_n^-=G_n^+$ on $B(x_0,2n)$. As a consequence,
the sequence $G_n^+$ is stabilizing on any bounded set, and has a limit $G$, which is
a spatial hypersurface, such that $\mJ^{++}(G)\cap S= \cup F_n$ and 
$\mJ^{--}(G)\cap S= \cup P_n$. To prove that $G$ is Cauchy, observe that the intersection
$\gamma\cap S$ of an inextensible causal curve $\gamma$ with the strip $S$ is bounded, hence the sequence $G_n^+\cap \gamma$ stabilizes and is equal to a single point.
\qed

\section{Density of completely uniform temporal functions in Cauchy temporal functions}\label{sec-unif}

We prove Theorem \ref{thm-dense}. 
We do not assume explicitly that $\mD$ has smooth boundary, but we assume that the conclusion of Theorem \ref{thm-two-unif} holds for $\mC$.
We fix a complete metric $g$.

Since Cauchy temporal functions are dense
in Cauchy causal functions, see \cite[Corollary 9]{BS2}, we can assume without loss of generality that 
$\tau$ is a Cauchy temporal function.

Notice that $au$ is a uniform temporal function if $u$ is and if $a>0$.
In view of this remark, it is enough to prove the result with $\epsilon=1$, which will simplify notations.

We will show the existence of a completely uniform function $v$ such that 
$v^{-1}([k,k+1])=\tau^{-1}([k,k+1])$ for each $k\in \Zm$, this implies that $\sup |v-\tau|\leq 1$.

By repeated applications of Lemma \ref{lem-mod}, we can replace the Cauchy  temporal function $\tau$ by a 
Cauchy temporal function $\tilde \tau$ such that $\tilde \tau^{-1}([k,k+1])=\tau^{-1}([k,k+1])$ for each $k\in \Zm$ and which moreover is $g$-uniform on 
$\tilde \tau^{-1}(B)$, with $B:=\{t\in \Rm: d(t,\Zm)<1/10\}$.

By repeated use of  Proposition \ref{prop-two-unif}, we  inductively  construct a sequence $v_n$, $n\geq 0$,
of smooth temporal functions and a sequence $h_n$, $n\geq 0$ of complete metrics with the following properties:
(a) $h_0=g$, $v_0=\tilde \tau$. (b) For each $n\geq 1$, $h_{n}=h_{n-1}$ and $v_{n}=v_{n-1}$ outside of 
$\tau^{-1}( ]-n, 1-n[\cup ]n-1,n[)$. (c) $v_n$ is $h_n$-uniform on $v_n^{-1}([-n,n]\cup B)$. (d) $v_n^{-1}([k,k+1])=\tau^{-1}([k,k+1])$ for each $n$ and $k$.
(e) $d_{h_n}(\tau^{-1}(k), \tau^{-1}(k+1))\geq 1/3$ for each $k$ in $\{-n, 1-n, \ldots, n-2,n-1\}$.
%
%
%

Since the sequences $v_n$ and $h_n$ stabilize on any compact set they converge
respectively to a smooth temporal function $v$ and a Riemannian metric $h$ such that $v$
is $h$-uniform. 
We just have to verify that $h$ is a complete metric. 
Let $\gamma$ be a curve of finite $h$-length. Since 
$d_{h}(\tau^{-1}(k), \tau^{-1}(k+1))\geq 1/3$ for each $k$, 
the curve $\gamma$ 
crosses only a finite number of integral level sets of $\tau$.
As a consequence, it is contained in $\tau^{-1}([-N,N])$ for some large $N$.
Then, the $h_N$-length of $\gamma$ is equal to its $h$-length, hence it is finite.
Since $h_N$ is complete, we deduce that $\gamma$ is relatively compact in $M$.

\section{Closed cone fields with smooth boundary}\label{sec-smooth}

The main goal of the present section is to finish the proof of Theorem \ref{thm-two-unif} by reducing the case with
smooth boundary to the non-degenerate case.

Let us first give an example illustrating the kind of difficulties which can 
occur for general globally hyperbolic cone fields.
We consider the standard Minkowski cone field $\mC_m$ on the plane,
and the closed set $\mD$ formed by the union of the three half lines starting at $0$ 
and directed by the vector $(1,1)$, $(-1,1)$ and $(0,-1)$.
The restriction $\mC$ to $\mD$ of $\mC_m$ is easily seen to be globally hyperbolic,
the second coordinate being a Cauchy temporal function $y$.
However, it does not admit a product structure, and not all level sets of $y_{|\mD}$ have the same topology (some have two points, some only one).

\begin{lem}
	Let $\mC$ be a locally solvable closed cone field with smooth boundary.
	Then at each boundary point $x$ of the domain, there exists a vector $v\in \mC(x)$
	which is tangent to the boundary and non-zero.	
\end{lem}

\begin{proof}
Let $f$ be a smooth function at $x$ such that, locally, $\mD=\{f\leq 0\}$,
and let $\tau$ be a local time function with $\tau(x)=0$.
Let $\gamma: ]-\epsilon, \epsilon[\lto M, \epsilon >0$ be a causal curve such that $\gamma(0)=x$,
parametrized by $\tau$ (meaning that $t=\tau \circ \gamma(t)$).
The local solvability of $\mC$ at $x$ implies the existence of such a curve.
We define the Clarke differential $\partial \gamma (0)$ of $\gamma$ at $0$ as the smallest closed convex
set containing the limits $\gamma'(t_n)$ for all sequences $t_n\lto 0$ of differentiability 
points of $\gamma$.
The function  $f\circ \gamma$ has a maximum at $t=0$,
and this implies that $df_x(\partial \gamma(0))$ contains $0$,
hence that $\partial \gamma (0)$ contains a vector $v$ tangent to the boundary.
Since $\tau \circ \gamma(t)=t$, we have $d\tau_x \cdot v=1$, hence $v\neq 0$.
Since $\gamma$ is causal, its Clarke differential is contained in $\mC(x)$,
see \cite[Section 2.2]{BS1}.
\end{proof}

The following result allows to deduce the study of globally hyperbolic cone fields with smooth boundary 
from the study of non-degenerate globally hyperbolic cone fields.

\begin{lem}\label{lem-tube}
	Let $\mC$ be a globally hyperbolic closed cone field with smooth  domain $\mD$. Let $\mC'$
	be an enlargement of $\mC$.
	Then there exists a smooth vector field $X$  contained in $\mC'$ (defined in a neighborhood of $\mD$) and a closed neighborhood with smooth boundary $\mF$ of $\mD$ (contained in the domain of $X$)
	such that $X$ is tangent to the boundaries of $\mD$ and $\mF$.
	
	If $\mC'$ is  hyperbolic, then the cone field $\mC_1$, defined to be equal to $\mC'$ on $\mD$
	and to the half line directed by $X$ on $\mF\setminus \mD$, is globally hyperbolic.
	
	The restriction of $\mC_1$ to the interior $U$ of $\mF$ is globally hyperbolic in $U$.
\end{lem}

\begin{proof}
Let $B$ be the boundary of the domain $\mD$.
By assumption it is a closed submanifold of $M$.
The enlargement $\mC'$ contains two smooth vector fields $X^{\pm}$ defined in a neighborhood of $\mD$ and pointing 
respectively strictly to the interior and strictly to the exterior of $\mD$.
Then, there exists a tubular neighborhood $W\approx B\times ]-1,1[$ of $B$, which is contained in the 
common domain of $X^{\pm}$, such that the leaves $B\times \{t\}$ are  transverse to $X^+$ and $X^-$,
and such that $B\times [-t,t]$ is closed in $M$ for each $t\in [0,1[$. There exists a smooth positive function $f:W\lto [0,1]$ such that
the vector field $X:= fX^-+(1-f)X^+$ is tangent to the leaves of $W$, it is obviously contained in $\mC'$.
We take $F:= \mD \cup (B\times [0, 1/2]) $, so that $U=\mD \cup (B\times [0, 1/2[)$.

Since the closed cone field $\mC_1$ is contained in the hyperbolic cone field $\mC'$, it is hyperbolic.
Moreover, it is locally solvable, hence globally hyperbolic.

Finally, if $K$ and $K'$ are compact subsets of $U$, then 
$
\mJ^+_{\mC_1|U}(K)\cap \mJ^-_{\mC_1|U}(K')=
\mJ^+_{\mC_1}(K)\cap \mJ^-_{\mC_1}(K')
$
is a compact subset of $U$, which implies the hyperbolicity of $\mC_{1|U}$.
\end{proof}

\begin{proof}[Proof of Corollary \ref{cor-prod}.]
We consider a globally hyperbolic enlargement $\mC'$ of $\mC$
such that $\tau$ is a Cauchy temporal function for $\mC'$, and apply Lemma \ref{lem-tube}.
Each integral curve of $X$ on $U$ intersects $N=\{x\in U : \tau(x)=0\}$ exactly once, and this defines
the map $\pi$. 
\end{proof}

\begin{proof}[Proof of Theorem \ref{thm-two-unif}.]
We consider a globally hyperbolic enlargement $\mC'$ of $\mC$ such that $H_0$ and $H_1$ are spatial for $\mC'$.
We then apply Lemma \ref{lem-tube}.
Since $\mC_{|U}$ is globally hyperbolic and non-degenerate on $U$, we can apply 
 Proposition \ref{prop-two-unif} to this cone field. We obtain a temporal function $\theta: U\lto \Rm$
which is uniform with respect to a complete metric $h$ on $U$. We now take a complete metric $g$ on $M$ 
which is equal to $h$ near $\mD$, and a smooth function $\tau$ on $M$ which is equal to $\theta$ near 
$\mD$.
\end{proof}

\bibliographystyle{amsplain}

\end{document}